\documentclass[12pt,leqno]{article}
\usepackage{amsmath,amssymb,amsthm,amscd,dsfont}
\numberwithin{equation}{section} \allowdisplaybreaks
\begin{document}
\newtheorem{theorem}{Theorem}[section]
\newtheorem{defin}{Definition}[section]
\newtheorem{prop}{Proposition}[section]
\newtheorem{corol}{Corollary}[section]
\newtheorem{lemma}{Lemma}[section]
\newtheorem{rem}{Remark}[section]
\newtheorem{example}{Example}[section]
\title{Hamiltonian vector fields on almost symplectic
manifolds} \author{{\small by}\vspace{2mm}\\Izu Vaisman}
\date{} \maketitle
{\def\thefootnote{*}\footnotetext[1]%
{{\it 2000 Mathematics Subject Classification: 53D15}.
\newline\indent{\it Key words and phrases}: Almost symplectic manifold; Hamiltonian vector field.}} \begin{center}
\begin{minipage}{12cm} A{\footnotesize BSTRACT. Let $(M,\omega)$ be an almost symplectic manifold ($\omega$ is a non degenerate, not closed, $2$-form). We say that a vector field $X$ of $M$ is locally Hamiltonian if $L_X\omega=0,d(i(X)\omega)=0$, and it is Hamiltonian if, furthermore, the $1$-form $i(X)\omega$ is exact. Such vector fields were considered in \cite{FS}, under the name of strongly Hamiltonian, and a corresponding action-angle theorem was proven. Almost symplectic manifolds may have few, non-zero, Hamiltonian vector fields or even none. Therefore, it is important to have examples and it is our aim to provide such examples here. We also obtain some new general results. In particular, we show that the locally Hamiltonian vector fields generate a Dirac structure on $M$ and we state a reduction theorem of the Marsden-Weinstein type. A final section is dedicated to almost symplectic structures on tangent bundles.} \end{minipage} \end{center} \vspace*{5mm}
\section{Introduction} All the objects that we consider are
assumed to be $C^\infty$-smooth and we follow the usual notation of differential geometric literature \cite{KN}.

The classical framework of Hamiltonian dynamics
is a symplectic manifold $(M^{2n},\omega)$,
where $\omega$ is a non degenerate, closed $2$-form
\cite{{LM},{OR}}. This framework was also extended to Poisson and Dirac manifolds \cite{C} and to similar structures on Lie
algebroids \cite{CMM}. In all these cases an integrability
condition that generalizes the closure of the symplectic form plays an essential role.

In \cite{FS}, the authors extend an important
theorem of Hamiltonian dynamics, existence of action-angle coordinates, to almost symplectic manifolds $(M^{2n},\omega)$,
where the $2$-form $\omega$ is still non degenerate but not closed.
I have no knowledge neither of earlier studies of
Hamiltonian fields on almost symplectic manifolds nor of
further developments, which, probably, is due to the fact that almost symplectic manifolds may have very few suitable Hamiltonian vector fields
(called strongly Hamiltonian in \cite{FS}), if at all; in \cite{FS} there are no concrete examples beyond the
general action-angle coordinates expression of the $2$-form $\omega$.

Nevertheless, some almost-symplectic manifolds may carry
interesting Hamiltonian fields. The aim of this note is to give such examples as well as some new general results concerning Hamiltonian vector fields on almost-symplectic manifolds. The latter results include the existence of a Dirac structure generated by locally Hamiltonian vector fields and a Marsden-Weinstein reduction theorem.
\section{Definitions and general results}
We define a suitable notion of Hamiltonian vector field of an almost symplectic manifold as follows.
\begin{defin}\label{defHas} {\rm Let $(M^{2n},\omega)$ be an almost symplectic manifold. A vector field $X$ on $M$ will be called a {\it Hamiltonian vector field} if
\begin{equation}\label{campH}
L_X\omega=0,\;i(X)\omega=-df\hspace{3mm}
(f\in C^\infty(M)),\end{equation}} where $L$ is the Lie derivative.
\end{defin}

In \cite{FS} the vector fields that satisfy (\ref{campH}) were called strongly Hamiltonian.
We will say that $X$ is the Hamiltonian vector field of $f$ and denote $X=X_f$. The function $f$ is a {\it Hamiltonian function} of $X$ and it is defined up to a constant. It follows easily that, if $f,h\in C^\infty(M)$ are Hamiltonian functions
with Hamiltonian fields $X_f,X_h$, the vector field
$$X_{fh}=fX_h+hX_f$$ is a Hamiltonian vector field with the product $fh$ as a Hamiltonian function. The algebra of the Hamiltonian functions on $M$ (with usual functions product) will be denoted by $\mathcal{H}(M,\omega)$.

Conditions (\ref{campH}) imply $di(X)\omega=0$. Accordingly, a vector field $X$ such that
\begin{equation}\label{camplocH}
di(X)\omega=0,\;L_X\omega=0\end{equation} will be called a {\it
locally Hamiltonian vector field}.
Due to the classical formula $L_X=di(X)+i(X)d$,
conditions (\ref{camplocH}) may be replaced by
\begin{equation}\label{camplocH1}
di(X)\omega=0,\;i(X)d\omega=0.\end{equation}

In the symplectic case $d\omega=0$ and we regain the classical definitions. Some of the well known symplectic properties hold in the general case too. For two locally Hamiltonian vector fields $X,Y$, the commutation formula
$$i([X,Y])=L_Xi(Y)-i(Y)L_X$$ yields
$$i([X,Y])\omega=-d(\omega(X,Y)).$$ Therefore, $[X,Y]$ is a
Hamiltonian vector field with Hamiltonian function
$\omega(X,Y)$. Furthermore, for $f,h\in\mathcal{H}(M,\omega)$, we get a skew-symmetric bracket
\begin{equation}\label{Poissbr}\{f,h\}=\omega(X_f,X_h)=X_fh=-X_hf,
\end{equation} such that
$$X_{\{f,h\}}=[X_f,X_h].$$
\begin{defin}\label{tgHas} {\rm A
vector $X_0\in T_{x_0}M$, $x_0\in M$,
will be called a {\it Hamiltonian tangent vector} if there
exists a locally Hamiltonian vector field $X$ defined on an
open neighborhood of $x_0$ such that
$X(x_0)=X_0$. A submanifold $N\subseteq M$ such that all its tangent vectors are Hamiltonian will be called a {\it Hamiltonian submanifold}.}\end{defin}

We will denote by $H_{x_0}M\subseteq T_{x_0}M$ the real linear subspace of
Hamiltonian tangent vectors at $x_0$. In the symplectic
case $H_{x_0}M=T_{x_0}M$ but, in the general case non-Hamiltonian tangent vectors may exist. By (\ref{camplocH1}), the condition $i(X_0)(d\omega)_{x_0}=0$ is a necessary condition for $X_0\in H_{x_0}M$, which, however, may not be sufficient. A better result is given by
\begin{prop}\label{necsufH1} The vector $X_0\in T_{x_0}M$ is a Hamiltonian tangent vector iff the system of partial differential equations
$$\omega^{ij}(d\omega)_{ihk}\frac{\partial f}{\partial x^j}=0,$$
where $(x^j)$ are local coordinates around $x_0$,
has local, differentiable solutions $f(x^j)$ such that $$\left.\frac{\partial f}{\partial x^i}\right|_{x_0}=\omega_{ij}(x_0)\xi^j_0,\hspace{4mm} X_0=\xi^j_0\left.\frac{\partial}{\partial x^j}\right|_{x_0}.$$\end{prop}
\begin{proof} The required initial conditions ensure that $X=\sharp_\omega df$ is an extension of $X_0$ such that $d(i(X)\omega)=0$ and the system of equations is the expression of the condition $i(X)d\omega=0$.\end{proof}

The necessary condition $i(X_0)(d\omega)_{x_0}=0$ implies
\begin{equation}\label{dH0}d\omega(X_{x_0},Y_{x_0},Z_{x_0})=0,\;\; \forall
X_{x_0},Y_{x_0},Z_{x_0}\in H_{x_0}M,\end{equation}
therefore, if $\iota:N\hookrightarrow M$ is a Hamiltonian submanifold, $d\iota^*\omega=0$.
Furthermore, (\ref{dH0}) implies the Jacobi identity for the bracket (\ref{Poissbr}), hence, (\ref{Poissbr}) is a Poisson bracket that defines a Poisson algebra structure on the subset $\mathcal{H}(M,\omega)$ of Hamiltonian functions.

Following is a more significant result
\begin{theorem}\label{Dirac} The field of planes $H_{x}M$ ($x\in M$) is a generalized foliation $H$, which, together with the restrictions of $\omega$ to the leaves of $H$ produces a Dirac structure $\mathcal{D}_\omega$ on $M$. \end{theorem}
\begin{proof} The distribution $H$ is locally spanned by the Lie algebra of locally Hamiltonian vector fields, hence, it is differentiable. Any $\omega$-preserving diffeomorphism of $M$ sends a locally Hamiltonian field to a locally Hamiltonian field. In particular this is true for the flow of a locally Hamiltonian field, whence we get $dim\,H_{exp(tX)(x)}=const.$ Under these conditions, $H$ is known to be integrable (e.g., Theorem 2.9", \cite{V1}). Obviously, the leaves of $H$ are Hamiltonian submanifolds and any connected Hamiltonian submanifold is contained in a leaf of $H$.

By results of \cite{C}, the integrability of $H$ and (\ref{dH0}) ensure that the field of subspaces of $TM\oplus T^*M$ defined by
\begin{equation}\label{Diracomega} \mathcal{D}_\omega(x)=
\{(X,\flat_{\varpi_x} X+\nu)\,/\,X\in H_x,\,\nu\in ann\,H_x\},
\end{equation} where $x$ runs in $M$, $\varpi_x$ is the restriction of $\omega_x$ to the leaf of $H$ through $x$ and $\flat_\varpi X$ is a leaf-wise form, is a Dirac structure \cite{C}.
\end{proof}

On an almost symplectic manifold $(M^{2n},\omega)$ the
following operators play an important role
$$\Lambda=i(\omega^{-1}),\;\delta=\star d\star,\;
\star\nu=[i(\sharp_\omega\nu)\omega^n]/n!\;\;(\nu\in\Omega^k(M))$$
($i$ denotes contraction by a multivector, $\omega^{-1}$ is the inverse bivector of $\omega$ and $\sharp_\omega=\flat_\omega^{-1}$,where $\flat_\omega$ is defined on vector fields by $\flat_\omega X=i(X)\omega)$). The forms $\nu$ such that $\Lambda\nu=0$ (in particular, all $0$-forms and $1$-forms) are called {\it primitive} or {\it effective} and any form can be presented by means of $\omega$ and of primitive forms (the Lepage decomposition
\cite{LM}).

The following proposition gives a characterization of
Hamiltonian fields on almost symplectic manifolds.
\begin{prop}\label{prpLib} For any almost symplectic manifold
$(M^{2n},\omega)$ one has \begin{equation}\label{Lib}
d\omega=\sigma\wedge\omega+\psi,
\;\;\sigma=\frac{1}{n-1}\delta\omega,\end{equation} where
$\psi$ is primitive and vanishes for $n=2$. A vector field $X$
on $M$ is locally Hamiltonian iff
\begin{equation}\label{eqhamX} \sigma(X)=0,\; \sigma\wedge
(i(X)\omega)-i(X)\psi=0,\;di(X)\omega=0.\end{equation} A function $f\in
C^\infty(M)$ is a Hamiltonian function iff
\begin{equation}\label{eqhamf} (\sharp_\omega\sigma)f=0,\;
\sigma\wedge df+i(\sharp_\omega(df))\psi=0.\end{equation}
\end{prop} \begin{proof} Formula (\ref{Lib}) is known \cite{{Lib},{LM}}).
The first condition (\ref{camplocH1}) is included in (\ref{eqhamX}) and via (\ref{Lib}) the second condition (\ref{camplocH1}) becomes
$$\sigma(X)\omega-\sigma\wedge(i(X)\omega)+i(X)\psi=0.$$
By applying $\Lambda$ to the above and since $\psi$ is a primitive form, we get $(n-1)\sigma(X)=0$ and (\ref{eqhamX}) is proven. Then, taking $X=-\sharp_\omega(df)$ we get the conclusion for Hamiltonian functions. \end{proof}
\begin{corol}\label{cazpsi0} If $n\geq3$, $\psi=0$ and $\sigma\neq0$ on a
dense subset of $M$, then, $\sigma$ is closed and a function
$f$ defined on a neighborhood where $\sigma=dt$ is Hamiltonian
on that neighborhood iff $f=f(t)$.\end{corol}
\begin{proof} The fact that $n\geq3$ and $\psi=0$ imply $d\sigma=0$
is known and shows that $M$ is a locally conformal symplectic
manifold (e.g., \cite{LM}). The second condition (\ref{eqhamf}) reduces to $\sigma\wedge df=0$ and it implies the first
condition (\ref{eqhamf}) because it implies that $\sigma$ is proportional to $df$ wherever $df\neq0$. Thus, $f$ is Hamiltonian iff $\sigma\wedge df=0$ and our hypotheses yield $df=s\sigma=sdt$ for a function $s\in C^\infty(M)$. This implies the required conclusion.\end{proof}
\begin{rem}\label{obsn2} {\rm The conclusion of Corollary
\ref{cazpsi0} also holds for $n=2$ if $d\sigma=0$ is added to
the hypotheses. Corollary \ref{cazpsi0} shows that the
locally conformal symplectic manifolds have few Hamiltonian
functions and all the Poisson brackets are zero.}\end{rem}

We end this section by refereing to Marsden-Weinstein reduction theory in the almost symplectic case, while assuming that the reader is familiar with the corresponding theory on symplectic manifolds (e.g., \cite{{LM},{OR}}).

We follow our paper \cite{V0}, where the interest was in the para-Hermitian case. The wording of the definition of {\it Hamiltonian actions} and {\it equivariant momentum maps} is the same as in symplectic geometry but, the notion of a Hamiltonian vector field is that of the almost symplectic case. Obviously, the orbits of a Hamiltonian action are Hamiltonian submanifolds.

Assume that we have a Hamiltonian action of the Lie group $G$ on the almost symplectic manifold $(M,\omega)$ that has an equivariant momentum map $\Phi:M\rightarrow\mathfrak{g}^*$, where $\mathfrak{g}^*$ is the dual of the Lie algebra of $G$. Let $\theta\in\mathfrak{g}^*$ be a non critical value of $\Phi$ and $N= \Phi^{-1}(\theta)$ be the corresponding $G_\theta$-invariant, level submanifold of $M$, where $G_\theta$ is the isotropy subgroup of $\theta$ under the coadjoint action of $G$. For $x\in N$, we will denote by $G(x),G_\theta(x)$ the $G$, respectively $G_\theta$, orbit of $x$. Then, $N\cap G(x)=G_\theta(x)$ and (like in the symplectic case \cite{LM}) $T_xN\perp_\omega T_x(G(x))$, therefore, $K=ker(\iota^*\omega_x)=T_xN\cap T_x(G(x))=T_x(G_\theta(x))$ ($\iota:N\hookrightarrow M$).
\begin{theorem}\label{redMW} With the notation above enabled, if the action of $G_\theta$ on $N$ is free and proper, there exists a reduced quotient manifold $Q$ with the projection $q:N\rightarrow Q$ and with a reduced almost symplectic structure $\varpi$ such that $q^*\varpi=\iota^*\omega$.\end{theorem}
\begin{proof} Since the action of $G$ is free and proper, the set of the $G_\theta$-orbits of the points of $N$ is the required quotient manifold $Q$ and we have the projection $q:N\rightarrow Q$. Furthermore, we see that the subbundle $K$ is tangent to the foliation $\mathcal{K}$ of $N$ by the connected components of the $G_\theta$-orbits. It follows that the local cross sections of $K$ are spanned by Hamiltonian vector fields and $L_X\omega=0$, $\forall X\in\Gamma K$. The same vector fields $X\in\Gamma K$ satisfy $i(X)\omega=0$ because $K=ker(\iota^*\omega_x)$. These two facts ensure the existence of the reduced $2$-form $\varpi$ on $Q$. Moreover, $K=ker(\iota^*\omega_x)$ implies the non degeneracy of $\varpi$.\end{proof}
\section{Examples of Hamiltonian vector fields}
In this section we give examples of Hamiltonian functions and vector fields on almost symplectic, non-symplectic manifolds.
\begin{example}\label{locprod} {\rm
Assume that the almost symplectic manifold $(M,\omega)$ has a locally product structure defined by the foliations $\mathcal{F}_1,\mathcal{F}_2$ with the local equations $x^i=const.$, $y^k=const.$, respectively, and that $$\omega=\omega_1+\omega_2,\;\;\omega_1=
\frac{1}{2}\varphi_{ij}(x)dx^i\wedge dx^j,\,\omega_2=\frac{1}{2}\psi_{kh}(y)dy^k\wedge dy^h,$$
where $\omega_1$ is an almost symplectic form in the coordinates $(x^i)$ and $\omega_2$ is a symplectic form in the coordinates $(y^k)$ (i.e., $d\omega_2=0$). Then, the function $f(x^i,y^k)$ is a Hamiltonian function on $M$ iff it is $\omega_1$-Hamiltonian as a function of $(x^i)$ and the Hamiltonian vector field of $f$ on $M$ is the sum of its $\omega_1$-Hamiltonian field with its $\omega_2$-Hamiltonian field.}
\end{example}
\begin{example}\label{funchammilt} {\rm \cite{V0} Take
$$M=\{(x^1,x^2,x^3,x^4)\,/\,
x^1>0,x^2>0\}\subseteq\mathds{R}^4$$ and $$\omega=x^1dx^2\wedge
dx^3+x^2dx^1\wedge dx^4.$$ This is a globally conformal
symplectic manifold where
$$\sigma=d(ln(x^1x^2)),\;\psi=0.$$ Hence, by Corollary \ref{cazpsi0}, the Hamiltonian functions
are the functions $f(t)$, $t=x^1x^2$. The Hamiltonian vector
field of such a function is $$X_f=\frac{\partial f}{\partial
t}\left(\frac{\partial}{\partial
x^3}+\frac{\partial}{\partial x^4}\right).$$

The Dirac structure (\ref{Diracomega}) of the present example is generated by cross sections of the form $$\mathcal{D}_\omega=\{(h\left(\frac{\partial}{\partial
x^3}+\frac{\partial}{\partial x^4}\right),\nu_1dx^1+\nu_2dx^2+ \varphi(dx^3-dx^4))\}$$ where $h,\nu_1,\nu_2,\varphi\in C^\infty(M)$.
In particular, we see that the notions of $\omega$-Hamiltonian function and $\mathcal{D}_\omega$-Hamiltonian function in the sense of \cite{C} are different. Indeed, for any function $l(x^1,x^2,x^3,x^4)$ such that $\partial l/\partial x^3= -\partial l/\partial x^4$ we have $$(h\left(\frac{\partial}{\partial
x^3}+\frac{\partial}{\partial x^4}\right),dl)\in\mathcal{D}_\omega$$ and $l$ is a $\mathcal{D}_\omega$-Hamiltonian function, while it may not be $\omega$-Hamiltonian.

The Abelian group $G=\mathds{R}$ has the following action on $M$:
$$ \tilde{x}^1=x^1,\,\tilde{x}^2=x^2,\,\tilde{x}^3=x^3+a,\,
\tilde{x}^4=x^4+a\;\;(a\in\mathds{R}).$$ The corresponding infinitesimal action of the natural basis of the Lie algebra $\mathds{R}$ of $G$ is the vector field $\partial/\partial
x^3+\partial/\partial x^4$. Therefore, the action is Hamiltonian and has the equivariant momentum map
$$\Phi(x^1,x^2,x^3,x^4)=x^1x^2.$$ Every value $t>0\in\mathds{R}$ is non-critical for $\Phi$ with the level manifold $N=\Phi^{-1}(t)$ defined in $M$ by the equation $x^1x^2=t(=const.)$ On $N$, we have the global coordinates $x^1>0,x^3,x^4$ and, if we denote $\iota:N\hookrightarrow M$, then
$$\iota^*\omega= -\frac{dx^1}{x^1}\wedge(dx^3-dx^4).$$ This also is the expression of the reduced form $\varpi$ of the quotient manifold $Q$ of $N$ by the orbits of the restriction of $G=\mathds{R}$ to $N$. More exactly, $(Q,\varpi)$ is symplectomorphic with the plane $(\mathds{R}^2,du\wedge dv)$ by $u=ln\,x^1,v=x^4-x^3$.}\end{example}
\begin{example}\label{Maltomega} {\rm On the manifold $M$ of Example \ref{funchammilt} take the almost symplectic form
$$\theta=dx^1\wedge dx^2+dx^1\wedge dx^3+x^1x^2dx^3\wedge dx^4.$$ Then,
$$d\theta=x^1dx^2\wedge dx^3\wedge dx^4+x^2dx^1\wedge dx^3\wedge dx^4$$ and the vector fields that satisfy the condition $i(X)d\theta=0$ are given by the formula
$$X=f\left(x^1\frac{\partial}{\partial x^1}- x^2\frac{\partial}{\partial x^2}\right).$$
But, the condition $d(i(X)\theta)=0$ holds only for $f=0$. Therefore, no non-zero locally Hamiltonian vector fields exists. In particular, at any point of $M$, there are vectors that satisfy the necessary condition of a Hamiltonian tangent vector but they are not such. Notice that for the $2$-form $\theta$, (\ref{Lib}) holds with
$$\sigma=\frac{dx^1}{x^1}+\frac{dx^2}{x^2}+\frac{dx^3}{x^2},\;\psi=0,$$
where $d\sigma\neq0$.}\end{example}
\begin{example}\label{idemdimn} {\rm Take
$\tilde{M}=M\times\mathds{R}^{2n-4}$ with the non degenerate
$2$-form
$$\tilde{\omega}=\omega
+e^{x^3x^4}(\sum_{h=1}^{n-2}dy^h\wedge dy^{n-2+h}),$$ where $M$ and $\omega$ are those defined in Example \ref{funchammilt}
and $y^s$ are the natural coordinates on $\mathds{R}^{2n-4}$. Then,
$$d\tilde{\omega}=d(x^3-x^4)\wedge dx^1\wedge dx^2 +e^{x^3x^4}d(x^3x^4)\wedge
(\sum_{h=1}^{n-2}dy^h\wedge dy^{n-2+h})$$
and the only vector field
$$X=\sum_{i=1}^4\xi^i\frac{\partial}{\partial x^i}
+\sum_{k=1}^{2n-4}\eta^k\frac{\partial}{\partial y^k}$$ that satisfies the condition $i(X)d\tilde{\omega}=0$ is $X=0$. Hence, $(\tilde{M},\tilde{\omega})$ has no nonzero, locally Hamiltonian vector fields and no non-constant Hamiltonian functions.

But, following Example \ref{locprod}, if we replace $\tilde{\omega}$ by
$$\bar{\omega}=\omega
+\sum_{h=1}^{n-2}dy^h\wedge dy^{n-2+h},$$
any function of the form $f(t,y^1,...,y^{2n-4})$ ($t=x^1x^2$) is Hamiltonian and its Hamiltonian field is
$$X_f=\frac{\partial f}{\partial
t}\left(\frac{\partial}{\partial
x^3}+\frac{\partial}{\partial x^4}\right)+X^y_{f}$$
where $X^y_{f}$ is the Hamiltonian field of $f$ with respect to the canonical symplectic form of $\mathds{R}^{2n-4}$.}\end{example}
\begin{example}\label{exgroup} {\rm It is known that the manifold $M=G_1\times G_2$, where $G_1=G_2=G$ is a Lie group endowed with a left-invariant pseudo-Riemannian metric $\gamma$, has a canonical para-Hermitian structure (e.g., \cite{EST}). The fundamental form of this structure is an almost symplectic structure on $M$ that may be defined as follows.

Put $$\pi_1(g_1,g_2)=g_1,\pi_2(g_1,g_2)=g_2,
\iota_1(g)=(g,e),\iota_2(g)=(e,g),$$ where $e$ is the unit of $G$ and, generally, attach the index $1,2$ to images by $\pi_1,\iota_1$, $\pi_2,\iota_2$ of objects of $G$. Then, let $(Y_i)$, $(\omega^i)$ be a basis of left invariant vector fields and the dual basis of left invariant forms of $G$. The announced almost symplectic structure is
$$\omega=\gamma_{ij}\omega_1^i\wedge\omega_2^j,
\hspace{5mm}\gamma_{ij}=\gamma(Y_i,Y_j)=const.$$
Using the Cartan equations $$d\omega^i=\frac{1}{2}(c^i_{jk}\omega^k\wedge\omega^j)$$ where $c^i_{jk}$ are the structure constants of the Lie algebra $\mathfrak{g}$ of $G$, it follows that $d\omega=0$ iff the group $G$ is Abelian, hence, if $G$ is not Abelian, $(M,\omega)$ is a non-symplectic, almost symplectic manifold.

Now, consider a vector field of the form
$$
Z=X_1+X_2=\xi^iY_{i1}+\xi^iY_{i2},$$
where $X=\xi^iY_i$, $\xi^i=const.$ is a left invariant vector field on $G$.

A straightforward calculation gives
$$i(Z)\omega=\gamma_{ij}(\xi^i\omega_2^j-\xi^j\omega_1^i),\;
di(Z)\omega=\frac{1}{2}\gamma_{ij}(\xi^ic^j_{hk} \omega_2^h\wedge\omega_2^k-\xi^jc^i_{hk} \omega_1^h\wedge\omega_1^k),$$
therefore, $di(Z)\omega=0$ iff $\gamma_{ij}\xi^ic^j_{hk}=0$. The meaning of this condition is that the element $X$ of the Lie algebra $\mathfrak{g}$ of $G$ that defines the vector filed $Z$ must be $\gamma$-orthogonal to the derived algebra $[\mathfrak{g},\mathfrak{g}]$.

The second condition required for $Z$ to be locally Hamiltonian is $i(Z)d\omega=0$ and may also be expressed using the Cartan equations. A straightforward examination of the result shows it to be equivalent with the following global property $$\gamma(ad\,X(U),V)+\gamma(U,ad\,X(V))=0,
\hspace{2mm}X,U,V\in\mathfrak{g},$$
which is further equivalent to $L_X\gamma=0$, where $X$ and $\gamma$ are seen as left invariant tensor fields on $G$ and $L$ is the Lie derivative.

The conclusion is that $X\in\mathfrak{g}$ defines a locally Hamiltonian vector field $Z=X_1+X_2$ on $M$ iff $X$ is $\gamma$-orthogonal to the derived algebra of $\mathfrak{g}$ and the right translations by the flow of $X$ preserve the left invariant metric $\gamma$.}\end{example}
\section{Structures on tangent bundles}
In this section we discuss some almost symplectic structures on a tangent bundle $M=TN\stackrel{\pi}{\rightarrow}N$, where $N$ is an $n$-dimensional manifold. For the geometry of tangent bundles we refer the reader to \cite{YI}; a brief survey can be found in \cite{Vtg}.

On $M$ one has the tangent structure tensor field $S\in End(TM)$,
\begin{equation}\label{defS}
S\frac{\partial}{\partial x^i} =\frac{\partial}{\partial y^i},\,S\frac{\partial}{\partial y^i} =0,\end{equation}
where $(x^i)$, $i=1,...,n$, are local coordinates on $N$ and $(y^i)$ are vector coordinates with respect to the bases $(\partial/\partial x^i)$ (equations (\ref{defS}) are invariant under coordinate changes on $N$).

One has two intrinsic lift operations from $N$ to $M$. Firstly, there exists a unique homomorphism of real tensor algebras $V:\mathcal{T}_x\rightarrow\mathcal{T}_y$ ($x\in N,y\in\pi^{-1}(x)$) such that
\begin{equation}\label{liftv}
V(1)=1,\,V(\alpha)=\pi^*\alpha,\,V(X)=S\mathcal{X},\end{equation}
where $\alpha\in T^*_xN,
\mathcal{X}\in T_yM,\pi_*\mathcal{X}=X$.
This homomorphism is called the vertical lift and, instead of the notation of (\ref{liftv}), the images will be denoted by an upper index $v$. Namely, one has
$$ \Theta^v( \mathcal{X}_1,...,\mathcal{X}_q,
\lambda_1,...,\lambda_p)=\Theta(\pi_*\mathcal{X}_1,...,\pi_*\mathcal{X}_q,
\alpha_{\lambda_1},...,\alpha_{\lambda_p}),$$ where $\mathcal{X}_i\in T_yM,\lambda_i\in T^*_yM$ and
$\alpha_\lambda\in T^*_xN$ is characterized by $\alpha_\lambda(X)=\lambda(S\mathcal{X})$ with $X,\mathcal{X}$ like in (\ref{liftv}). If $\Theta$ is a form, then, $\Theta^v=\pi^*\Theta$. The name comes from the fact that $V(T_xN)=S(T_xM)=\mathfrak{V}_y$, where $\mathfrak{V}$ is the tangent bundle of the fibers of $M$, usually called the vertical bundle.

Secondly, there exists a homomorphism $C$ of real linear spaces, with images denoted by an upper index $c$, from the space of tensor fields of type $(p,q)$ on $N$ to the similar space on $M$, called the complete lift, such that
\begin{equation}\label{Ctensor}(P\otimes Q)^c=P^c\otimes Q^v+P^v\otimes Q^c\end{equation} and which is defined on functions, vector fields and $1$-forms in the following way. If $f\in C^\infty(N)$,
$$f^c(y)=y(f)=y^i\frac{\partial f}{\partial x^i},\hspace{3mm}(y\in M).$$ If $X\in\Gamma TN$, $X^c$ is the tangent vector field of the lifted flow $(exp\,tX)_*$; in local coordinates the complete lift is given by
$$ X=\xi^i(x^j)\frac{\partial}{\partial x^i},\;
X^c=\xi^i\frac{\partial}{\partial x^i}+y^j\frac{\partial \xi^i}{\partial x^j}\frac{\partial}{\partial y^i}.$$ If $\alpha$ is a $1$-form on $N$ then
$$\alpha^c(X^v)=(\alpha(X))^v,\,\alpha^c(X^c)=(\alpha(X))^c,$$
which, in local coordinates gives
$$ \alpha=\alpha_i(x^j)dx^i,\,\alpha^c=
y^j\frac{\partial\alpha_i}{\partial x^j}dx^i+\alpha_idy^i.$$

Now, let us assume that $N$ has an almost symplectic structure
$$\sigma=\frac{1}{2}\sigma_{ij}(x^k)dx^i\wedge dx^j.$$ Then, we have the following results.
\begin{prop}\label{sigmac} The complete lift $\sigma^c$ is an almost symplectic structure on $M$. If $X_f$ is a $\sigma$-Hamiltonian vector field on $N$ $(f\in C^\infty(N))$, then, the complete lift $X_f^c$ is $\sigma^c$-Hamiltonian for the function $f^c\in C^\infty(M)$. If the Lie group $G$ has a $\sigma$-Hamiltonian action on $N$ with an equivariant momentum map $\Phi$, then, the differential of this action is a $\sigma^c$-Hamiltonian action of $G$ on $M$, which has the equivariant momentum map $\Phi^c$.\end{prop}
\begin{proof} Using (\ref{Ctensor}) and the local coordinate expressions of $f^c,X^c,\alpha^c$, we get
$$\sigma^c=\frac{1}{2}y^k\frac{\partial\sigma_{ij}}{\partial y^k}dx^i\wedge dx^j+\sigma_{ij}dx^i\wedge dy^j,$$ which proves the first assertion. Furthermore \cite{YI}, it is known that
$\forall X\in\Gamma TN,\Theta\in\Omega^s(N)$, one has
\begin{equation}\label{resYI}(d\Theta)^c=d(\Theta^c), \;(i(X)\Theta)^c=i(X^c)\Theta^c, \;L_{X^c}\Theta^c=(L_X\Theta)^c.\end{equation} The remaining assertions of the proposition are straightforward consequences of (\ref{resYI}) used for $X=X_f,\Theta=\sigma$. The infinitesimal transformations of $G$ on $M$ are the complete lifts of the infinitesimal transformations on $N$. The complete lift of a function extends to vector valued functions, which gives the meaning of the complete lift of the momentum map. Notice that, if used for open neighborhoods in $N$, the proven results show that the complete lift of a locally $\sigma$-Hamiltonian, vector field is a locally $\sigma^c$-Hamiltonian, vector field.
\end{proof}

In order to obtain other interesting almost symplectic structures on tangent bundles we shall assume that $N$ is endowed with a (pseudo-)Riemannian metric $\gamma$ and that a choice of a horizontal bundle $ \mathfrak{H}$ was made, i.e., we have a decomposition
$$TM=\mathfrak{H}\oplus\mathfrak{V}.$$
Then, we also have the horizontal lift $(X\in T_xN)\mapsto(X^h\in T_yM)$ defined by the conditions $X^h\in\mathfrak{H}_y,\pi_*X^h=X$. This lift also extends to tensors \cite{YI}.

On $M$, we define an almost symplectic form $\omega$ {\it associated} to $(\gamma,\mathfrak{H})$ by
$\omega|_{ \mathfrak{H}}=0$, $\omega|_{ \mathfrak{V}}=0$ and
$$\omega_y( \mathcal{X},\mathcal{Y} )=-\omega_y( \mathcal{Y},\mathcal{X} )=\gamma_{\pi(y)}(X,Y),$$ for
$\mathcal{X}=X^h,\,\mathcal{Y}=Y^v,\;y\in M,\, X,Y\in T_{\pi(y)}N$.

We shall also need the {\it associated metric} on $M$ by the equalities
$$g(X^h,Y^h)=g(SX^h,SY^h)=\gamma(X,Y)=\omega(X^h,SY^h),\,
g(X^h,Y^v)=0.$$

The tensors $\omega,g$ may be expressed by means of local coordinates as follows \cite{VCz}. $ \mathfrak{V}$ has the local bases $\partial/\partial y^i$ and $ \mathfrak{H}$ has the bases given by the horizontal lifts of $\partial/\partial x^i$:
\begin{equation}\label{bazah}X_i=\frac{\partial}{\partial x^i}-t^j_i(x^k,y^l)\frac{\partial}{\partial y^j},\end{equation}
where $-t^j_i$ are the {\it coefficients of the non-linear connection} $\mathfrak{H}$. The annihilators of $ \mathfrak{V}, \mathfrak{H}$ have the corresponding dual bases
\begin{equation}\label{cobazah}dx^i,\; \theta^i=dy^i+t^i_jdx^j.\end{equation} With respect to these bases we get
\begin{equation}\label{omegag} \omega=\gamma_{ij}dx^i\wedge\theta^j,\;
g=\gamma_{ij}dx^i\otimes dx^j+\gamma_{ij}\theta^i\otimes \theta^j,\end{equation} where $\gamma_{ij}$ are the local components of the metric tensor $\gamma$.
\begin{prop}\label{propHv} {\rm1}. A vertical vector field $X^v$ is locally Hamiltonian with respect to the structure $\omega$ associated to $(\gamma,\mathfrak{H})$ iff
$d(i(X^v)\omega)=0$. {\rm2}. For any function $f\in C^\infty(N)$, the function $\pi^*f\in C^\infty(M)$ is $\omega$-Hamiltonian with the vertical Hamiltonian vector field $X^v=-\sharp_\omega d(\pi^*f)$.\end{prop}
\begin{proof}
1. Since $ \mathfrak{V}$ is $\omega$-Lagrangian we get
\begin{equation}\label{auxrm1}i(X^v)d\omega(Y^v,Z^v)=d\omega(X^v,Y^v,Z^v)=0.
\end{equation} We also have
\begin{equation}\label{auxrm2}
i(X^v)d\omega(Y^v,Z^h)=d\omega(X^v,Y^v,Z^h)=0.
\end{equation}
Indeed, it suffices to check this for $X^v=\partial/\partial y^i,Y^v=\partial/\partial y^j,X_k)$ and the result follows from (\ref{bazah}), (\ref{cobazah}) and (\ref{omegag}). The same local expressions yield
$$ i(X^v)d\omega(Y^h,Z^h)=d\omega(X^v,Y^h,Z^h)=-d(i(X^v)\omega)(Y^h,Z^h),
$$
hence, $d(i(X^v)\omega)=0$ implies $i(X^v)d\omega=0$ and we are done.

2. With the first conclusion proven, the only fact we still have to check for the second conclusion is the verticality of the vector field $\sharp_\omega d(\pi^*f)$. Since $\mathfrak{V}$ is $\omega$-Lagrangian, verticality is equivalent to $\omega(\sharp_\omega d\pi^*f,Y^v)=Y^v\pi^*f=0$, $\forall Y^v$, which is true.
\end{proof}
\begin{rem}\label{obsHv} {\rm With the notation of Proposition \ref{propHv}, it is easy to check that
$i(X^v)\omega=-(\flat_\gamma X^v)\circ S$ and that the indicated Hamiltonian field of $\pi^*f$ is also equal to $X^v=\sharp_g[(d\pi^*f)\circ S']$, where $g$ is the associated metric and $S'\in End\,TM$ is zero on $ \mathfrak{H}$ and is defined on $ \mathfrak{V}$ by the conditions \begin{equation}\label{S'}S'X^v\in\mathfrak{H},\, SS'X^v=X^v.\end{equation}}\end{rem}
\begin{corol}\label{verteinH} {\rm1.} With respect to the associated almost symplectic form $\omega$, every vertical tangent vector of $M$ is a Hamiltonian tangent vector equal to the point-value of a vertical, local Hamiltonian vector field. {\rm2.} The vertical, locally Hamiltonian vector fields on $M$ are in a one-to-one correspondence with the closed $1$-forms on $N$.\end{corol}
\begin{proof} 1. Take $y\in M$ and $Z^v_y\in\mathfrak{V}_y$. Locally, extend the tangent covector $i(Z^v_y)\omega_y\in ann\,\mathfrak{V}_y$ to a closed $1$-form $\alpha\in ann\,\mathfrak{V}$. The assertion of the corollary holds because the equation $i(Z^v)\omega=\alpha$ has a vertical solution $Z^v$ that equals $Z^v_y$ at $y$. 2. The mapping $Z^v
 \rightarrow\alpha=i(Z^v)\omega$ is a bijection $\Gamma\mathfrak{V}\rightarrow\Gamma(ann\,\mathfrak{V})$ and $Z^v$ satisfies the condition of part 1 of Proposition \ref{propHv} iff the corresponding form $\alpha$ is the pullback of a closed $1$-form of $N$.\end{proof}

In order to understand other $\omega$-Hamiltonian conditions we will use the associated metric $g$.
Since $ \mathfrak{H}\perp_g\mathfrak{V}$ we have the so-called second canonical connection \cite{VCz} defined by
$$\begin{array}{l}D_{X^h}Y^h=pr_{ \mathfrak{H}}\nabla_{X^h}Y^h,\,D_{X^h}Y^v= pr_{ \mathfrak{V}}[X^h,Y^v],\vspace{2mm}\\  D_{X^v}Y^h= pr_{ \mathfrak{H}}[X^v,Y^h],\,D_{X^v}Y^v=pr_{ \mathfrak{V}}\nabla_{X^v}Y^h,\end{array}$$
where $\nabla$ is the Levi-Civita connection of $g$.

Connection $D$ is characterized by the preservation of $\mathfrak{H}$ and $\mathfrak{V}$, by the preservation of $g|_{\mathcal{H}},g|_{\mathfrak{V}}$ along $\mathfrak{H}, \mathfrak{V}$ and by the following expression of the torsion:
$$ T_D(Z_1,Z_2)=-pr_{\mathfrak{V}}
[pr_{ \mathfrak{H}}Z_1,pr_{ \mathfrak{H}}Z_2]=-R_{ \mathfrak{H}}(Z_1,Z_2),\;Z_1,Z_2\in\Gamma TM,$$
where $R_{ \mathfrak{H}}$ denotes the Ehressmann curvature of the distribution $ \mathfrak{H}$ seen as an	 Ehressmann connection.

On the other hand, we recall the following general expression of the exterior differential of a $2$-form $\omega$ by means of a connection $D$:
$$ d\omega(X,Y,Z)=\sum_{Cycl(X,Y,Z)}[D_X\omega(Y,Z)+\omega(T_D(X,Y),Z)].
$$

Now, we will prove the following proposition.
\begin{prop}\label{propHh} Let $f\in C^\infty(M)$ be a first integral of the horizontal distribution $ \mathfrak{H}$. Then, $grad_g\,f=\sharp_g df$ is a vertical vector field, hence, $X^h=-S'grad_g\,f$ is horizontal. Furthermore, if: {\rm1)} $X^h$ belongs to the kernel of the Ehressmann curvature $R_{ \mathfrak{H}}$ and {\rm2)} $X^h$ is $g$-orthogonal to the image of $R_{ \mathfrak{H}}$, then $f$ is a Hamiltonian function with the horizontal Hamiltonian vector field $X^h$.\end{prop}
\begin{proof} $S'$ is defined by (\ref{S'}). We begin by discussing conditions ensuring that a horizontal vector field $X^h$ is a local-Hamiltonian field.
Since $d\omega$ is skew-symmetric and only the point-values of the arguments count, equalities (\ref{auxrm1}), (\ref{auxrm2}) and part 1 of Corollary \ref{verteinH} yield
$$i(X^h)d\omega(Y^v,Z^v)=0,\, i(X^h)d\omega(Y^h,Z^v)=0.$$
On the other hand, using the definition of the connection $D$ and the expression of its torsion , we get
$$ i(X^h)d\omega(Y^h,Z^h)=-
\sum_{Cycl(X,Y,Z)}\omega(R_{ \mathfrak{H}}(X^h,Y^h),Z^h).
$$
Then, using the relation between $\omega$ and $g$ included in the definition of $g$, we see that any horizontal vector field $X^h$ that satisfies conditions 1), 2) stated in the proposition also satisfies the condition $i(X^h)d\omega(Y^h,Z^h)=0$, therefore, we have $i(X^h)d\omega(Y,Z)=0$, $Y,Z\in\Gamma TM$.

Furthermore, for the horizontal vector $X^h$ we have $i(X^h)\omega=(\flat_g(SX^h)$. Thus, in order to have a horizontal, locally Hamiltonian vector field we must add the condition $d[(\flat_g(SX^h)]=0$, equivalently, $SX^h=-\sharp_g(df)$for some local differentiable functions $f$.

Now, we translate the previous conditions into conditions that characterize $\omega$-Hamiltonian functions with a horizontal Hamiltonian vector field. Firstly, we have to ask the vector field $grad_gf$ to be vertical, which ensures the existence of a horizontal field $X^h$ such that $SX^h=-\sharp_g(df)$. This happens iff $f$ is a first integral of the horizontal distribution $ \mathfrak{H}$. Then, if we also ask the corresponding field $X_h$ to satisfy conditions 1), 2), $f$ is Hamiltonian with Hamiltonian field $X^h$. This completes the proof of the proposition.\end{proof}

We shall end by stressing the following point: the results concerning the horizontal case may be used for any vector field $X$ on $M$ that is never vertical (thus, never zero as well), for instance, for the complete lift of a tangent vector field of $N$ without vanishing points.

Indeed, let us denote by $\{X\}$ the line spanned by $X$. Then,
$$ \mathfrak{V}\cap\{X\}^{\perp_\omega}=
(\mathfrak{V}\oplus\{X\})^{\perp_\omega}$$ is an isotropic subbundle of rank $n-1$ and, by a well known result of symplectic geometry (e.g., \cite{V2}, Theorem 2.2.4), there exist isotropic subbundles $\mathfrak{S}$ of rank $n-1$ such that
$$TM=[(\mathfrak{V}\cap\{X\}^{\perp_\omega})\oplus\mathfrak{S}]
\oplus_{\perp_\omega}\Pi.$$ Above, $\oplus_{\perp_\omega}$ denotes the direct sum of symplectic, $\omega$-orthogonal subbundles and $\Pi$ is any complementary subbundle of $\mathfrak{V}\cap\{X\}^{\perp_\omega}$ in $\mathfrak{V}\oplus\{X\}$ with $\omega|_\Pi$ non degenerate. Of course, we may choose $\Pi$ such that $X\in\Pi$ and, then, $\bar{\mathfrak{H}}=\mathfrak{S}\oplus\{X\}$ is $\omega$-Lagrangian and such that $TM=\bar{\mathfrak{H}}\oplus
\mathfrak{V}$.

It is easy to check that the almost symplectic structure $\omega$ is also associated to the pair $(\gamma,\bar{ \mathfrak{H}})$, but, the associated metric $g$ is replaced by a metric $\bar{g}$. Proposition \ref{propHh} and the connected results may be used	 for the second canonical connection of the metric $\bar g$.

Moreover, we do not have to start with the pair $(\gamma,\mathfrak{H})$, and we may apply Propositions \ref{propHv}, \ref{propHh} for any almost symplectic structure $\omega$ on $M=TN$ such that the vertical subbundle $\mathfrak{V}$ is Lagrangian. We just have to choose an auxiliary Lagrangian subbundle $\mathfrak{H}$ that is complementary to $\mathfrak{V}$ and an auxiliary metric $\gamma$ on $N$.
{\small
Department of Mathematics,
University of Haifa, Israel, vaisman@math.haifa.ac.il
}
\end{document}